\documentclass[12pt]{amsart}%
\usepackage{amsmath}
\usepackage{amsfonts}
\usepackage{amssymb}
\usepackage{amstext}
\usepackage{graphicx}%
\providecommand{\U}[1]{\protect\rule{.1in}{.1in}}
\providecommand{\U}[1]{\protect \rule{.1in}{.1in}}
\providecommand{\U}[1]{\protect \rule{.1in}{.1in}}
\newtheorem{theorem}{Theorem}[section]
\newtheorem{lemma}{Lemma}[section]

\newtheorem{corollary}{Corollary}[section]

\newtheorem{definition}{Definition}[section]

\numberwithin{equation}{section}

\theoremstyle{remark}
\newtheorem{remark}{Remark}[section]
\numberwithin{equation}{section}
\begin{document}
\title[The CR Torsion Flow ]{Short-Time Existence Theorem for the CR Torsion Flow }
\author{Shu-Cheng Chang$^{1\ast}$}
\address{$^{1}$Department of Mathematics and Taida Institute for Mathematical Sciences,
National Taiwan University, Taipei 10617, Taiwan\\
Current address : Yau Mathematical Sciences Center, Tsinghua University,
Beijing, China}
\email{scchang@math.ntu.edu.tw }
\author{Chin-Tung Wu$^{2\ast}$}
\address{$^{2}$Department of Applied Mathematics, National PingTung University,
PingTung 90003, Taiwan, R.O.C.}
\email{ctwu@mail.nptu.edu.tw }
\thanks{$^{\ast}$Research supported in part by MOST of Taiwan.}
\subjclass{53C21, 32G07.}
\keywords{Ricci flow, Torsion flow, CR Einstein-Hilbert functional, Lax-Milgram theorem,
CR Paneitz operator, CR pluriharmonic operator, Tanaka-Webster scalar
curvature, Pseudohermitian torsion, Cauchy-Kovalevskaya theorem.}
\maketitle

\begin{abstract}
In this paper, we study the torsion flow which is served as the CR analogue of
the Ricci flow in a closed pseudohermitian manifold. We show that there exists
a unique smooth solution to the CR torsion flow in a small time interval with
the CR pluriharmonic function as an initial data. In spirit, it is the CR
analogue of the Cauchy-Kovalevskaya local existence and uniqueness theorem for
analytic partial differential equations associated with Cauchy initial value problems.

\end{abstract}

\section{Introduction}

One of the goals for differential geometry and geometric analysis is to
understand and classify the singularity models of a nonlinear geometric
evolution equation, and to connect it to the existence problem of geometric
structures on manifolds. For instance in 1982, R. Hamilton (\cite{h3})
introduced the Ricci flow. Then by studying the singularity models (\cite{h2},
\cite{pe1}, \cite{pe2}, \cite{pe3}) of Ricci flow, R. Hamilton and G. Perelman
solved the Thurston geometrization conjecture and Poincare conjecture for a
closed $3$-manifold in 2002.

It is natural then to investigate a corresponding geometrization problem of
closed CR manifolds by finding a CR analogue of the Ricci flow. More
precisely, let us recall that a strictly pseudoconvex CR structure on a
pseudohermitian $(2n+1)$-manifold $(M,\xi,J,\theta)$ is given by a cooriented
plane field $\xi=\ker\theta$, where $\theta$ is a contact form, together with
a compatible complex structure $J$. Given this data, there is a natural
connection, the so-called Tanaka-Webster connection or pseudohermitian
connection. We denote the torsion of this connection by $A_{J,\theta}$, and
the Tanaka-Webster curvature by $W$. Then, following this direction, we
consider the following so-called torsion flow on $(M,\xi,J,\theta
)\times\lbrack0,T):$
\begin{equation}
\left\{
\begin{array}
[c]{l}%
\frac{\partial J}{\partial t}=2A_{J,\theta},\\
\frac{\partial\theta}{\partial t}=-2W\theta,\\
\theta=e^{2\lambda}\theta_{0}%
\end{array}
\right.  \label{1}%
\end{equation}
It seems to us that the torsion flow (\ref{1}) is the right CR analogue of the
Ricci flow. \ It is the negative gradient flow of CR Einstein-Hilbert
functional (\ref{2015E}) in view of (\ref{2015D}). More precisely, let
$\left\{  T,Z_{\alpha},Z_{\bar{\beta}}\right\}  $ be a frame of $TM\otimes
\mathbb{C}$ with $\xi\otimes\mathbb{C=}T^{1,0}\oplus T^{0,1}$, where
$\{Z_{\alpha}\}$ is any local frame of $T^{1,0}$, and $Z_{\bar{\beta}%
}=\overline{Z_{\beta}}\in T^{0,1}$. Then $\left\{  \theta,\theta^{\alpha
},\theta^{\bar{\beta}}\right\}  $, the coframe dual to $\left\{  T,Z_{\alpha
},Z_{\bar{\beta}}\right\}  $, satisfies
\[
d\theta=ih_{\alpha\bar{\beta}}\theta^{\alpha}\wedge\theta^{\bar{\beta}},
\]
where $h_{\alpha\bar{\beta}}$ is a positive definite Levi metric. By the
Gram-Schmidt process we can always choose $Z_{\alpha}$ such that
$h_{\alpha\bar{\beta}}=\delta_{\alpha\bar{\beta}}$; throughout this paper, we
shall take such a frame. Then we write $J=i\theta^{\alpha}\otimes Z_{\alpha
}-i\theta^{\overline{\alpha}}\otimes Z_{\overline{\alpha}}$. Define
\[
A_{J,\theta}:=A^{\bar{\beta}}{}_{\alpha}Z_{\bar{\beta}}\otimes\theta^{\alpha
}+A^{\beta}{}_{\bar{\alpha}}Z_{\beta}\otimes\theta^{\bar{\alpha}},
\]
and
\[
E=E_{\alpha}{}^{\bar{\beta}}\theta^{\alpha}\otimes Z_{\bar{\beta}}%
+E_{\bar{\alpha}}{}^{\beta}\theta^{\bar{\alpha}}\otimes Z_{\beta},
\]
and consider the general CR flow on $(M,\xi,J,\theta)\times\lbrack0,T)$ given
by
\begin{equation}
\left\{
\begin{array}
[c]{l}%
\frac{\partial J(t)}{\partial t}=2E(t),\\
\frac{\partial\theta}{\partial t}=2\eta(t)\theta(t).
\end{array}
\right.  \label{2018aaa}%
\end{equation}
The CR Einstein-Hilbert functional is defined by
\begin{equation}
\mathcal{E}(J(t),\theta(t))=\int_{M}W(t)d\mu. \label{2015E}%
\end{equation}
Here $d\mu=\theta\wedge d\theta^{n}$ is the volume form. Since $\frac
{\partial}{\partial t}d\mu=4\eta(t)d\mu$, it follows that
\begin{equation}%
\begin{array}
[c]{ccl}%
\frac{d}{dt}\mathcal{E}(J(t),\theta(t)) & = & -\int_{M}\{(A^{\bar{\alpha}}%
{}_{\beta}E^{\beta}{}_{\bar{\alpha}}+A^{\alpha}{}_{\bar{\beta}}E_{{}}%
^{\bar{\beta}}{\alpha})-2\eta W\}d\mu\\
& = & -2\int_{M}\Vert A_{J,\theta}\Vert^{2}d\mu-2\int_{M}W^{2}d\mu\\
& \leq & 0
\end{array}
\label{2015D}%
\end{equation}
if we put
\[
E=A_{J,\theta}\text{\ \ \ \textrm{and\ }\ \ }d\eta(t)=-W(t).
\]

Unlike the Ricci flow, it is not clear that one can apply the so-called
DeTurck's trick (\cite{de}) to prove the short-time solution of the torsion
flow. It is therefore natural to consider the following version of CR torsion
flow ( see section $4$ for details)
\begin{equation}
\left\{
\begin{array}
[c]{l}%
\frac{\partial J}{\partial t}=2A_{J,\theta},\\
\frac{\partial\theta}{\partial t}=-2W\theta,\ \\
\text{\textrm{\ }}\theta=e^{2\gamma}\theta_{0},\ \gamma(0)=\gamma
_{0};\ \overset{0}{P_{\beta}}(\gamma_{0})=0=W^{\perp}(x,0).
\end{array}
\right.  \label{2015BBB}%
\end{equation}
on $M\times\lbrack0,T)$ in which \textbf{the initial Tanaka-Webster scalar
curvature }$W(x,0)$ w.r.t. $e^{\gamma_{0}}\theta_{0}$ and $\gamma_{0}$ are
\textbf{pluriharmonic }with respect to $\theta_{0}$ (see section $3$) and%
\[
W^{\perp}(x,t):=W(x,t)-W^{\ker}(x,t).
\]
Here $W^{\ker}$ is the CR-pluriharmonic function with respect to $\theta_{0}$
and $\xi=\ker\theta_{0}.$ In fact, if $M$ is a hypersurface in $\mathbb{C}%
^{n+1}$, i.e. $M=\partial\Omega$ for a strictly pseudoconvex bounded domain
$\Omega$ in $\mathbb{C}^{n+1}$, then for any pluriharmonic function
$u:\mathbf{U}\rightarrow\mathbb{R}$ ($\partial\overline{\partial}u=0$) with
$f:=u|_{M}$, it follows that $f$ is a CR pluriharmonic function (see
Definition \ref{4}). Moreover, for a simply connected $\mathbf{U}%
\subset\overline{\Omega}$, there exists a holomorphic function $w$ in
$\mathbf{U}$ such that $u=\mathrm{Re}(w)$.

In this paper, we show that there exists a unique smooth solution to the CR
torsion flow (\ref{2015BBB}) in a small time interval with the CR
pluriharmonic function as an initial data. In the above spirit, it is the CR
analogue of the Cauchy-Kovalevskaya local existence and uniqueness theorem for
analytic partial differential equations associated with Cauchy initial value problems.

We first show that there always exist such $\theta_{0}$ in $\xi$
and$\ \gamma_{0}$ so that
\begin{equation}
\overset{0}{P_{\beta}}(\gamma_{0})=0=W^{\perp}(x,0). \label{2018AAA1}%
\end{equation}
It is well-known, for the CR Yamabe problem on $(M,\xi)$, that there always
exists a contact form $\overline{\theta}_{0}$ of constant Tanaka-Webster
scalar curvature in a contact class $[\widetilde{\theta}_{0}]$ with $\xi
=\ker\widetilde{\theta}_{0}$. Furthermore, we can write
\[
\overline{\theta}_{0}=e^{2\lambda_{0}}\widetilde{\theta}_{0}%
\]
for some function $\lambda_{0}$. Rewrite
\[
\overline{\theta}_{0}=e^{2\lambda_{0}}\widetilde{\theta}_{0}=e^{2(\lambda
_{0})^{\ker}}(e^{2(\lambda_{0})^{\perp}}\widetilde{\theta}_{0}),
\]
we can choose the particular background contact form $\theta_{0}$ as
\begin{equation}
\theta_{0}:=e^{2(\lambda_{0})^{\perp}}\widetilde{\theta}_{0}%
\text{\ \ and\ \ \ \ }\gamma_{0}:=(\lambda_{0})^{\ker}.\text{\ }
\label{2018AAA}%
\end{equation}
It follows that the initial contact form
\[
\theta(0)=\overline{\theta}_{0}=e^{2\gamma_{0}}\theta_{0}%
\]
satisfies (\ref{2018AAA1}). Then we are in the situation as in (\ref{2015BBB}).

In this paper, we first follow Hamilton's original ideas (\cite{h1}) to prove
the following short-time existence result for the CR torsion
flow~(\ref{2015BBB}) in a closed pseudohermitian $(2n+1)$-manifold
$(M,\xi,\theta_{0}).$ In the appendix, we also give an alternative proof from
the point view of Lax-Milgram theorem (\cite{gt}) for $n=1$.

\begin{theorem}
\label{TA} Let $(M,\xi)$ be a closed CR $(2n+1)$-manifold. \textit{Let }%
$J_{0}$\textit{\ be any }$C^{\infty}$\textit{\ smooth oriented }%
$CR$\textit{\ structure compatible with }$\xi$. If there exist $\gamma_{0}$
and \textit{ }$C^{\infty}$\textit{\ smooth oriented contact form }$\theta_{0}$
with $\xi=\ker\theta_{0}$ so that
\[
\overset{0}{P_{\beta}}(\gamma_{0})=0\ \ \mathrm{\ and\ }\ \ \ W^{\perp
}(x,0)=0.
\]
Then there exists $\delta>0$ and a unique smooth solution $(J(t),\theta(t)\,)$
to the CR torsion flow~(\ref{2015BBB}) on $(M,\xi,J_{0},\theta_{0}%
)\times\lbrack0,\delta)$ such that $(J(0),\theta(0)\,)=(J_{0},e^{2\gamma_{0}%
}\theta_{0})$.
\end{theorem}

In particular, if we choose the special forms for $\gamma_{0}$ and $\theta
_{0}$ as in \textit{(\ref{2018AAA})}, then we have the following short-time
existence result.

\begin{corollary}
\label{C1} Let $(M,\xi)$ be a closed CR $(2n+1)$-manifold. \textit{Let }%
$J_{0}$\textit{\ be any }$C^{\infty}$\textit{\ smooth oriented }%
$CR$\textit{\ structure compatible with }$\xi$ and choose \textit{\ a
}$C^{\infty}$\textit{\ smooth oriented contact form }$\theta_{0}$ and a smooth
function $\gamma_{0}$ as in \textit{(\ref{2018AAA}). }Then there exists
$\delta>0$ and a unique smooth solution $(J(t),\theta(t)\,)$ to the CR torsion
flow~(\ref{2015BBB}) on $(M,\xi,J_{0},\theta_{0})\times\lbrack0,\delta)$ such
that $(J(0),\theta(0)\,)=(J_{0},e^{2\gamma_{0}}\theta_{0})$ with
\[
\overset{0}{P_{\beta}}(\gamma_{0})=0\ \ \mathrm{\ and\ }\ \ \ W^{\perp
}(x,0)=0.
\]

\end{corollary}

\begin{remark}
\label{R1} Let $(M,\xi)$ be a closed CR $3$-manifold with a
\textbf{pseudo-Einstein contact form }$\overline{\theta}_{0}$ of\textbf{\ }the
pluriharmonic scalar curvature (in particular, the constant scalar curvature).
Then there exists a pluriharmonic function $\gamma_{0}$ such that
$\overline{\theta}_{0}=e^{\gamma_{0}}\theta_{0}$ for some
\textbf{pseudo-Einstein contact form }$\theta_{0}$ due to J. Lee (\cite{l1})
and K. Hirachi (\cite{hi}). We will study the problem of asymptotic
convergence of solutions of the CR torsion flow (\ref{2015BBB}) on this
special situation for the forthcoming topic.
\end{remark}

The torsion flow greatly simplifies if the torsion vanishes. This only happens
in very special setups. Indeed, CR $3$-manifolds with vanishing torsion are
$K$-contact, meaning that the Reeb vector field is a Killing vector field for
the contact metric $g=\frac{1}{2}d\theta+\theta^{2}.$ In general, one can
still hope that the torsion flow improves properties of the contact manifold
underlying the CR-manifold. It is the case in a closed homogeneous
pseudohermitian $3$-manifold whose Lie algebra is isomorphic to $su(2)$%
.\ \ The torsion flow reduces to an ODE if we start with some appropriate
initial conditions and we are able to obtain the long-time existence and
asymptotic convergence of solutions for the (normalized) torsion flow
(\ref{2015BBB}) in this special case. More precisely, for any choice of
homogeneous complex structure $J$ on $SU(2)$, the solution of the normalized
torsion flow exists for all times and converges to the unique standard
CR-structure $J_{\mathrm{can}}$. These computations illustrate the behavior of
the torsion flow in special cases, and in these cases the torsion flow behaves
as can be expected from a Ricci-like flow. We refer to our previous paper
\cite{ckw} for more details.

In general, unlike the Hamilton Ricci flow, the problem of asymptotic
convergence of solutions of the CR torsion flow is widely open in closed CR
$3$-manifolds. \ The structure of CR torsion solitons may be a necessary step
in understanding the asymptotic convergence of solutions of the CR torsion
flow. Indeed, one expects CR torsion solitons to model singularity formations
of the CR torsion flow. In the joint works with H.-D. Cao and C.-W. Chen
(\cite{cacc1}, \cite{cacc2}), we investigate the geometry and classification
of closed three-dimensional CR Yamabe and Torsion solitons. We obtain a
classification theorem of complete three-dimensional CR Yamabe solitons of
vanishing torsion which are the same as torsion solitons. Furthermore, by
deriving the CR Hamilton's type Harnack quantity, we are able to show that any
closed three-dimensional CR torsion soliton must be the standard Sasakian
space form (\cite{t}).

We conclude this introduction with a brief plan of the paper. In Section~$2$,
we survey basic notions in CR geometry. In section $3,\ $we give some
preliminary results concerning the short time existence of the torsion flow,
in particular the evolution equation of the general CR flow (\ref{2018aaa})
and the invariance property for the CR Paneitz operator as in Lemma \ref{lTA}.
In section $4,$ We prove the main Theorem \ref{TA} as well as Corollary
\ref{C1}. In the appendix, we give an alternative proof via the
Lions-Lax-Milgram theorem.

\noindent\textbf{Acknowledgements} Part of the project was done during the
visit of the first named author to Yau Mathematical Sciences Center, Tsinghua
University. He would like to express his thanks for the warm hospitality. We
also thank Professor Otto van Koert for useful comments which leads this
current paper possible.

\section{Preliminaries}

In this section, we introduce some basic notions from pseudohermitian geometry
as in \cite{l1} and \cite{l2}, and refer to these papers for proofs and more references.

\begin{definition}
Let $M$ be a smooth manifold and $\xi\subset TM$ a subbundle. A \textbf{CR
structure} on $\xi$ consists of an endomorphism $J:\xi\rightarrow\xi$ with
$J^{2}=-id$ such that the following integrability condition holds :

(i) \  If $X,Y\in\xi$, then so is $[JX,Y]+[X,JY]$.

(ii)  $J([JX,Y]+[X,JY])=[JX,JY]-[X,Y]$.
\end{definition}

The CR structure $J$ can be extended to $\xi\otimes{\mathbb{C}}$, which we can
then decompose into the direct sum of eigenspaces of $J$. The eigenvalues of
$J$ are $i$ and $-i$, and the corresponding eigenspaces will be denoted by
$T^{1,0}$ and $T^{0,1}$, respectively. The integrability condition can then be
reformulated as
\[
X,Y\in T^{1,0} \text{ implies } [X,Y]\in T^{1,0}.
\]

Now consider a closed $2n+1$-manifold $M$ with a cooriented contact structure
$\xi=\ker\theta$. This means that $\theta\wedge d\theta^{n}\neq0$. The
\textbf{Reeb vector field} of $\theta$ is the vector field $T$ uniquely
determined by the equations
\[
{\theta}(T)=1,\quad\text{and}\quad d{\theta}(T,{\cdot})=0.
\]

A \textbf{pseudohermitian manifold} is a triple $(M^{2n+1},\theta,J),$ where
$\theta$ is a contact form on $M$ and $J$ is a CR structure on $\ker\theta$.
The \textbf{Levi form} $\left\langle \ ,\ \right\rangle $ is the Hermitian
form on $T^{1,0}$ defined by
\[
H(Z,W)=\left\langle Z,W\right\rangle =-i\left\langle d\theta,Z\wedge
\overline{W}\right\rangle .
\]

We can extend this Hermitian form $\left\langle \ ,\ \right\rangle $ to
$T^{0,1}$ by defining $\left\langle \overline{Z},\overline{W}\right\rangle
=\overline{\left\langle Z,W\right\rangle }$ for all $Z,W\in T^{1,0}$.
Furthermore, the Levi form naturally induces a Hermitian form on the dual
bundle of $T^{1,0}$, and hence on all induced tensor bundles.

We now restrict ourselves to strictly pseudoconvex manifolds, or in other
words compatible complex structures $J$. This means that the Levi form induces
a Hermitian metric $\langle\cdot,\cdot\rangle_{J,{\theta}}$ by
\[
\langle V,U\rangle_{J,{\theta}}=d{\theta}(V,JU).
\]
The associated norm is defined as usual: $|V|_{J,\theta}^{2}=\langle
V,V\rangle_{J,{\theta}}$. It follows that $H$ also gives rise to a Hermitian
metric for $T^{1,0}$, and hence we obtain Hermitian metrics on all induced
tensor bundles. By integrating this Hermitian metric over $M$ with respect to
the volume form $d\mu=\theta\wedge d\theta^{n} $, we get an $L^{2}$-inner
product on the space of sections of each tensor bundle.

The \textbf{pseudohermitian connection} or \textbf{Tanaka-Webster connection}
of $(J,\theta)$ is the connection $\nabla$ on $TM\otimes\mathbb{C}$ (and
extended to tensors) given in terms of a local frame $\{Z_{\alpha}\}$ for
$T^{1,0}$ by%

\[
\nabla Z_{\alpha}=\omega_{\alpha}{}^{\beta}\otimes Z_{\beta},\quad\nabla
Z_{\bar{\alpha}}=\omega_{\bar{\alpha}}{}^{\bar{\beta}}\otimes Z_{\bar{\beta}%
},\quad\nabla T=0,
\]
where $\omega_{\alpha}{}^{\beta}$ is the $1$-form uniquely determined by the
following equations:%

\[%
\begin{array}
[c]{rcl}%
d\theta^{\beta} & = & \theta^{\alpha}\wedge\omega_{\alpha}{}^{\beta}%
+\theta\wedge\tau^{\beta},\\
\tau_{\alpha}\wedge\theta^{\alpha} & = & 0,\\
\omega_{\alpha}{}^{\beta}+\omega_{\bar{\beta}}{}^{\bar{\alpha}} & = & 0.
\end{array}
\]

Here $\tau^{\alpha}$ is called the \textbf{pseudohermitian torsion}, which we
can also write as
\[
\tau_{\alpha}=A_{\alpha\beta}\theta^{\beta}.
\]
The components $A_{\alpha\beta}$ satisfy
\[
A_{\alpha\beta}=A_{\beta\alpha}.
\]
We often consider the \textbf{torsion tensor} given by
\[
A_{J,\theta}=A^{\alpha}{}_{\bar{\beta}}Z_{\alpha}\otimes\theta^{\bar{\beta}%
}+A^{\bar{\alpha}}{}_{\beta}Z_{\bar{\alpha}}\otimes\theta^{\beta}.
\]

We now consider the curvature of the Tanaka-Webster connection in terms of the
coframe $\{ \theta=\theta^{0},\theta^{\alpha},\theta^{\bar{\beta}}\} $. The
second structure equation gives
\[%
\begin{split}
\Omega_{\beta}{}^{\alpha}  &  =\overline{\Omega_{\bar{\beta}}{}^{\bar{\alpha}%
}}=d\omega_{\beta}{}^{\alpha}-\omega_{\beta}{}^{\gamma}\wedge\omega_{\gamma}%
{}^{\alpha},\\
\Omega_{0}{}^{\alpha}  &  =\Omega_{\alpha}{}^{0}=\Omega_{0}{}^{\bar{\beta}%
}=\Omega_{\bar{\beta}}{}^{0}=\Omega_{0}{}^{0}=0.
\end{split}
\]

In \cite[Formulas 1.33 and 1.35]{we}, Webster showed that the curvature
$\Omega_{\beta}{}^{\alpha}$ can be written as
\begin{equation}%
\begin{array}
[c]{c}%
\Omega_{\beta}{}^{\alpha}=R_{\beta}{}^{\alpha}{}_{\rho\bar{\sigma}}%
\theta^{\rho}\wedge\theta^{\bar{\sigma}}+W_{\beta}{}^{\alpha}{}_{\rho}%
\theta^{\rho}\wedge\theta-W^{\alpha}{}_{\beta\bar{\rho}}\theta^{\bar{\rho}%
}\wedge\theta+i\theta_{\beta}\wedge\tau^{\alpha}-i\tau_{\beta}\wedge
\theta^{\alpha},
\end{array}
\label{222}%
\end{equation}
where the coefficients satisfy
\[%
\begin{array}
[c]{c}%
R_{\beta\bar{\alpha}\rho\bar{\sigma}}=\overline{R_{\alpha\bar{\beta}\sigma
\bar{\rho}}}=R_{\bar{\alpha}\beta\bar{\sigma}\rho}=R_{\rho\bar{\alpha}%
\beta\bar{\sigma}},\ \ W_{\beta\bar{\alpha}\gamma}=W_{\gamma\bar{\alpha}\beta
}.
\end{array}
\]
In addition, by \cite[(2.4)]{l2} the coefficients $W_{\alpha}{}^{\beta}%
{}_{\rho}$ are determined by the torsion,
\[
W_{\alpha}{}^{\beta}{}_{\rho}=A_{\alpha\rho,}{}^{\beta}.
\]
Contraction of (\ref{222}) yields%
\[%
\begin{split}
\Omega_{\alpha}{}^{\alpha}=d\omega_{\alpha}{}^{\alpha} &  =R_{\rho\bar{\sigma
}}\theta^{\rho}\wedge\theta^{\bar{\sigma}}+W_{\alpha}{}^{\alpha}{}_{\rho
}\theta^{\rho}\wedge\theta-W_{\overline{\alpha}}{}^{\overline{\alpha}}{}%
_{\bar{\rho}}\theta^{\bar{\rho}}\wedge\theta\\
&  =R_{\rho\bar{\sigma}}\theta^{\rho}\wedge\theta^{\bar{\sigma}}+A_{\alpha
\rho}{}^{\alpha}\theta^{\rho}\wedge\theta-A_{\bar{\alpha}\bar{\rho}}{}%
^{\bar{\alpha}}\theta^{\bar{\rho}}\wedge\theta
\end{split}
\]

We will denote components of covariant derivatives by indices preceded by a
comma. For instance, we write $A_{\alpha\beta,\gamma}$. Here the indices
$\{0,\alpha,\bar{\beta}\}$ indicate derivatives with respect to $\{T,Z_{\alpha
},Z_{\bar{\beta}}\}$. For derivatives of a scalar function, we will often omit
the comma. For example, $\varphi_{\alpha}=Z_{\alpha}\varphi,\ \varphi
_{\alpha\bar{\beta}}=Z_{\bar{\beta}}Z_{\alpha}\varphi-\omega_{\alpha}%
{}^{\gamma}(Z_{\bar{\beta}})Z_{\gamma}\varphi,\ \varphi_{0}=T \varphi$ for a
(smooth) function $\varphi$.

For a real-valued function $\varphi$, the \textbf{subgradient} $\nabla
_{b}\varphi$ is defined as the unique vector field $\nabla_{b}\varphi\in\xi$
such that $\left\langle Z,\nabla_{b}\varphi\right\rangle =d\varphi(Z)$ for all
vector fields $Z$ tangent to the contact distribution $\xi$. Locally
$\nabla_{b}\varphi=\varphi_{\alpha}^{\alpha Z}+\varphi^{\bar{\alpha}}%
Z_{\bar{\alpha}}$. Define the \textbf{sublaplacian} $\Delta_{b}$ by
\[
\Delta_{b}\varphi=\varphi_{\alpha}{}^{\alpha}+\varphi_{\bar{\alpha}}{}%
^{\bar{\alpha}}.
\]

To consider smoothness for functions on strongly pseudo-convex manifolds, we
recall below what the Folland-Stein space is $S_{k,p}.$ Let $D$ denote a
differential operator acting on functions. We say $D$ has weight $m,$ denoted
$w(D)=m,$ if $m$ is the smallest integer such that $D$ can be locally
expressed as a polynomial of degree $m$ in vector fields tangent to the
contact bundle $\xi.$ We define the Folland-Stein space $S_{k,p}$ of functions
on $M$ by
\begin{equation}
S_{k,p}=\{f\in L^{p}\mid Df\in L^{p}\text{ whenever }w(D)\leq k\}.\label{KKK}%
\end{equation}

\noindent We define the $L^{p}$ norm of $\nabla_{b}f,$ $\nabla_{b}^{2}f$, ...
to be ($\int|\nabla_{b}f|^{p}\theta\wedge d\theta^{n})^{1/p},$ ($\int
|\nabla_{b}^{2}f|^{p}\theta\wedge d\theta^{n})^{1/p},$ $...,$ respectively, as
usual. So it is natural to define the $S_{k,p}$-norm $||f||_{S_{k,p}}$ of
$f\in S_{k,p}$ as follows:%
\[
||f||_{S_{k,p}}\equiv(\sum_{0\leq j\leq k}||\nabla_{b}^{j}f||_{L^{p}}%
^{p})^{1/p}.
\]

\noindent In particular, we denote $S_{k}:=S_{k,2}$ for$\ p=2$. The function
space $S_{k,p}$ with the above norm is a Banach space for $k\geq0,$
$1<p<\infty.$ There are also embedding theorems of Sobolev type. For instance,
$S_{2,2}\subset S_{1,4}$ (for $\dim M$ $=$ $3$). We refer the reader to, for
instance, \cite{fs} for more discussions on these spaces. Finally, we also
need commutation relations (\cite[Equation 2.15]{l2}), sometimes called Ricci
identities. Let $\varphi$ be a scalar function and $\sigma=\sigma_{\alpha
}\theta^{\alpha}$ be a $\left(  1,0\right)  $ form, then we have%

\[%
\begin{array}
[c]{rcl}%
\varphi_{\alpha\beta} & = & \varphi_{\beta\alpha},\\
\varphi_{\alpha\bar{\beta}}-\varphi_{\bar{\beta}\alpha} & = & ih_{\alpha
\overline{\beta}}\varphi_{0},\\
\varphi_{0\alpha}-\varphi_{\alpha0} & = & A_{\alpha\beta}\varphi_{\bar{\beta}%
},\\
\sigma_{\alpha,0\beta}-\sigma_{\alpha,\beta0} & = & \sigma_{\alpha,\bar
{\gamma}}A_{\gamma\beta}-\sigma_{\gamma}A_{\alpha\beta,\bar{\gamma}},\\
\sigma_{\alpha,0\bar{\beta}}-\sigma_{\alpha,\bar{\beta}0} & = & \sigma
_{\alpha,\gamma}A_{\bar{\gamma}\bar{\beta}}+\sigma_{\gamma}A_{\bar{\gamma}%
\bar{\beta},\alpha},
\end{array}
\]
and
\[%
\begin{array}
[c]{ccl}%
\sigma_{\alpha,\beta\gamma}-\sigma_{\alpha,\gamma\beta} & = & iA_{\alpha
\gamma}\sigma_{\beta}-iA_{\alpha\beta}\sigma_{\gamma},\\
\sigma_{\alpha,\bar{\beta}\bar{\gamma}}-\sigma_{\alpha,\bar{\gamma}\bar{\beta
}} & = & ih_{\alpha\overline{\beta}}A_{\bar{\gamma}\bar{\rho}}\sigma_{\rho
}-ih_{\alpha\overline{\gamma}}A_{\bar{\beta}\bar{\rho}}\sigma_{\rho},\\
\sigma_{\alpha,\beta\bar{\gamma}}-\sigma_{\alpha,\bar{\gamma}\beta} & = &
ih_{\beta\overline{\gamma}}\sigma_{\alpha,0}+R_{\alpha\bar{\rho}}{}_{\beta
\bar{\gamma}}\sigma_{\rho}.
\end{array}
\]
We also mention some curvature identities. The ones that are relevant for us
are the \textbf{contracted CR Bianchi identities}.
\[%
\begin{split}
R_{\rho\bar{\sigma},\gamma}-R_{\gamma\bar{\sigma},\rho} &  =iA_{\alpha\gamma
,}{}^{\alpha}h_{\rho\bar{\sigma}}-iA_{\alpha\rho,}{}^{\alpha}h_{\gamma
\bar{\sigma}},\\
W_{,\gamma}-R_{\gamma\bar{\sigma},}{}^{\bar{\sigma}} &  =i(n-1)A_{\alpha
\gamma,}{}^{\alpha},\\
R_{\rho\bar{\sigma},0} &  =A_{\alpha\rho,}{}^{\alpha}{}_{\bar{\sigma}}%
+A_{\bar{\beta}\bar{\sigma},}{}^{\bar{\beta}}{}_{\rho},\\
W_{,0} &  =A_{\alpha\rho,}{}^{\alpha\rho}+A_{\bar{\beta}\bar{\sigma},}{}%
^{\bar{\beta}\bar{\sigma}}.
\end{split}
\]

\section{Evolution Equations under the CR Curvature Flow}

In this section, we give some preliminary results concerning the short time
existence of the torsion flow. Let $\theta(t)$ be a family of smooth contact
forms and $J(t)$ be a family of CR structures on $(M,J,\theta)$. We consider
the following general CR flow on a closed pseudohermitian $(2n+1)$-manifold
$(M,J,\theta)\times\lbrack0,T)$ :
\begin{equation}
\left\{
\begin{array}
[c]{l}%
\frac{\partial J}{\partial t}=2E,\\
\frac{\partial\theta}{\partial t}=2\eta(t)\theta(t).
\end{array}
\right.  \label{PTF}%
\end{equation}
Here $J=i\theta^{\alpha}\otimes Z_{\alpha}-i\theta^{\overline{\alpha}}\otimes
Z_{\overline{\alpha}}$ and $E=E_{\alpha}{}^{\overline{\beta}}\theta^{\alpha
}\otimes Z_{\overline{\beta}}+E_{\overline{\alpha}}{}^{\beta}\theta
^{\overline{\alpha}}\otimes Z_{\beta}$.

We start by deriving some evolution equations under the general flow
(\ref{PTF}) before specifying to the torsion flow, for which $E=A_{J}$ (the
torsion tensor), and $\eta=-W$ (the Webster curvature). All computations will
be done in a local frame. Fix a unit-length local frame $\{Z_{\alpha}\}$ and
let $\{\theta^{\alpha}\}$ be its dual admissible $1$-form. Let $Z_{\alpha
(t)},\theta^{\alpha}(t)$ denote a unit-length frame and dual admissible
$1$-form with respect to $(J(t),\theta(t))$. Since $\theta^{\alpha}%
(Z_{\beta(t)})$ is a positive real function, we can write $\overset
{\centerdot}{Z_{\alpha}}=F_{\alpha}{}^{\beta}Z_{\beta}+G_{\alpha}{}%
^{\overline{\beta}}Z_{\overline{\beta}}$ where $F_{\alpha}{}^{\beta}$ are real
and $G_{\alpha}{}^{\overline{\beta}}$ are complex. The fact that
$Z_{\alpha(t)}$ is an orthonormal frame means that
\[
\delta_{\alpha\beta}=-id\theta(t)(Z_{\alpha(t)}\wedge Z_{\overline{\beta}%
(t)}).
\]
By differentiating and substituting the above expression for $\overset
{\centerdot}{Z_{\alpha}},$ we obtain $F_{\alpha}{}^{\beta}=-\eta\delta
_{\alpha}^{\beta}.$ By differentiating $J(t)Z_{\alpha(t)}=iZ_{\alpha(t)}$ we
find
\[
0=\overset{\centerdot}{J}Z_{\alpha}+J\overset{\centerdot}{Z_{\alpha}%
}-i\overset{\centerdot}{Z_{\alpha}}=2E_{\alpha}{}^{\overline{\beta}%
}Z_{\overline{\beta}}-2iG_{\alpha}{}^{\overline{\beta}}Z_{\overline{\beta}},
\]
so
\[
\overset{\centerdot}{Z_{\alpha}}=-\eta Z_{\alpha}-iE_{\alpha}{}^{\overline
{\beta}}Z_{\overline{\beta}}.
\]
Now differentiate the identities%
\[
d\theta=ih_{\alpha\bar{\beta}}\theta^{\alpha}\wedge\theta^{\overline{\beta}%
},\text{ \ }\theta^{\alpha}(Z_{\beta(t)})=\delta_{\beta}^{\alpha},\text{
\ \textrm{and} \ }\theta^{\alpha}(Z_{\overline{\beta}(t)})=0,
\]
to deduce that
\begin{equation}
\overset{\centerdot}{{\theta}}{{^{\alpha}}}=2i\eta^{\alpha}\theta+\eta
\theta^{\alpha}-iE^{\alpha}{_{\overline{\beta}}}\theta^{\overline{\beta}%
}.\label{24}%
\end{equation}
Now we differentiate (\ref{24}) to obtain%
\begin{equation}
d\overset{\centerdot}{{\theta}}{{^{\alpha}}}=\overset{\centerdot}{{\theta}%
}{{^{\gamma}}\wedge}{\omega_{\gamma}}^{\alpha}+\theta^{\gamma}{\wedge}%
\overset{\centerdot}{{\omega}}{_{\gamma}}^{\alpha}+\overset{\centerdot}{{A}%
}_{\overline{\alpha}\overline{\gamma}}{\theta}{\wedge}{\theta^{\overline
{\gamma}}+{A}}_{\overline{\alpha}\overline{\gamma}}\overset{\centerdot
}{{\theta}}{{\wedge}{\theta^{\overline{\gamma}}+{A}_{\overline{\alpha
}\overline{\gamma}}\theta}{\wedge}}\overset{\centerdot}{{\theta}}%
{{^{\overline{\gamma}}}.}\label{25}%
\end{equation}
Since we will derive an identity involving tensors, we will take an adapted
frame satisfying ${\omega_{\gamma}}^{\alpha}=0$ at a point. Plug in (\ref{25})
and consider the ${\theta}{\wedge}{\theta^{\overline{\gamma}}}$ terms to
obtain
\begin{equation}
\overset{\centerdot}{{A}}_{\overline{\alpha}\overline{\gamma}}=-2(i\eta
_{\overline{\alpha}\overline{\gamma}}+\eta{A}_{\overline{\alpha}%
\overline{\gamma}})-iE{_{\overline{\alpha}\overline{\gamma}}},_{0}.\label{26}%
\end{equation}
On the other hand, contracting (\ref{25}) with $Z_{\beta}$ and then
contracting with $h^{\beta\overline{\alpha}},$ computing modulo $\theta
^{\gamma}$ yields%
\[
\overset{\centerdot}{{\omega}}{_{\alpha}}^{\alpha}=i(A{_{\alpha}}%
^{\overline{\gamma}}E{_{\overline{\gamma}}}^{\alpha}+A{_{\overline{\alpha}}%
}^{\gamma}E{_{\gamma}}^{\overline{\alpha}}+\eta{_{\overline{\alpha}}%
}^{\overline{\alpha}})\theta-[(n+2)\eta_{\overline{\alpha}}+iE{_{\overline
{\gamma}\overline{\alpha}}},^{\overline{\gamma}}]{\theta^{\overline{\alpha}}%
}\text{ }\mod \theta^{\gamma}.
\]
Since $\overset{\centerdot}{{\omega}}{_{\alpha}}^{\alpha}$ is pure imaginary,
we have
\begin{equation}%
\begin{array}
[c]{lll}%
\overset{\centerdot}{{\omega}}{_{\alpha}}^{\alpha} & = & i(A{_{\alpha}%
}^{\overline{\gamma}}E{_{\overline{\gamma}}}^{\alpha}+A{_{\overline{\alpha}}%
}^{\gamma}E{_{\gamma}}^{\overline{\alpha}}+\Delta_{b}\eta)\theta\\
&  & +[(n+2)\eta_{\alpha}-iE{_{\gamma\alpha}},^{\gamma}]{\theta^{\alpha
}-[(n+2)\eta_{\overline{\alpha}}+iE{_{\overline{\gamma}\overline{\alpha}}%
},^{\overline{\gamma}}]\theta^{\overline{\alpha}}.}%
\end{array}
\label{27}%
\end{equation}
Differentiate the structure equation  with respect to $t$ and consider only
the $\theta^{\rho}\wedge\theta^{\bar{\sigma}}$ terms. This gives%
\begin{equation}%
\begin{array}
[c]{ccl}%
\overset{\centerdot}{R}_{\rho\bar{\sigma}} & = & -(A{_{\alpha}}^{\overline
{\gamma}}E{_{\overline{\gamma}}}^{\alpha}+A{_{\overline{\alpha}}}^{\gamma
}E{_{\gamma}}^{\overline{\alpha}}+\Delta_{b}\eta)h_{\rho\bar{\sigma}}-2\eta
R_{\rho\bar{\sigma}}\\
&  & -[(n+2)\eta_{\rho}-iE{_{\gamma\rho}},^{\gamma}],_{\overline{\sigma}%
}{-[(n+2)\eta_{\overline{\sigma}}+iE{_{\overline{\gamma}\overline{\sigma}}%
},^{\overline{\gamma}}]},_{\rho}%
\end{array}
\label{28}%
\end{equation}
After contracting with $h^{\rho\bar{\sigma}}$ we get%
\begin{equation}%
\begin{array}
[c]{ccl}%
\overset{\centerdot}{W} & = & i(E{_{\gamma\alpha}},^{\gamma\alpha
}-{E{_{\overline{\gamma}\overline{\alpha}}},^{\overline{\gamma}\overline
{\alpha}}})-n(A{_{\alpha}}^{\overline{\gamma}}E{_{\overline{\gamma}}}^{\alpha
}+A{_{\overline{\alpha}}}^{\gamma}E{_{\gamma}}^{\overline{\alpha}})\\
&  & -[2(n+1)\Delta_{b}\eta+2W\eta]\\
& = & 2\operatorname{Re}\left(  iE{_{\gamma\alpha}},^{\gamma\alpha
}-nA{_{\alpha}}^{\overline{\gamma}}E{_{\overline{\gamma}}}^{\alpha}\right)
-[2(n+1)\Delta_{b}\eta+2{W\eta].}%
\end{array}
\label{eq:variationW}%
\end{equation}

Recall that the transformation law of the connection under a change of
pseudohermitian structure was computed in \cite[Sec. 5]{l1}. Let $\hat{\theta
}=e^{2f}\theta$ be another pseudohermitian structure. Then we can define an
admissible coframe by $\hat{\theta}^{\alpha}=e^{f}(\theta^{\alpha}%
+2if^{\alpha}\theta)$. With respect to this coframe, the connection $1$-form
and the pseudohermitian torsion are given by%
\begin{equation}
\label{30}%
\begin{split}
\widehat{{\omega}}{_{\beta}}^{\alpha}  &  ={\omega_{\beta}}^{\alpha
}+2(f_{\beta}\theta^{\alpha}-f^{\alpha}\theta_{\beta})+\delta_{\beta}^{\alpha
}(f_{\gamma}\theta^{\gamma}-f^{\gamma}\theta_{\gamma})\\
&  \phantom{=}+i(f^{\alpha}{}_{\beta}+f_{\beta}{}^{\alpha}+4\delta_{\beta
}^{\alpha}f_{\gamma}f^{\gamma})\theta,
\end{split}
\end{equation}
and
\begin{equation}
\widehat{{A}}{_{\alpha\beta}=}e^{-2f}({A_{\alpha\beta}+2i}f_{\alpha\beta
}-4if_{\alpha}f_{\beta}), \label{31}%
\end{equation}
respectively. Thus the Webster curvature transforms as
\begin{equation}
\widehat{W}=e^{-2f}(W-2(n+1)\Delta_{b}f-4n(n+1)f_{\gamma}f^{\gamma}).
\label{32}%
\end{equation}
Here covariant derivatives on the right side are taken with respect to the
pseudohermitian structure $\theta$ and an admissible coframe $\theta^{\alpha}%
$. Note also that the dual frame of $\{ \hat{\theta},\hat{\theta}^{\alpha
},\hat{\theta}^{\overline{\alpha}}\}$ is given by $\{ \widehat{T},\widehat
{Z}_{\alpha},\widehat{Z}_{\overline{\alpha}}\}$, where%

\[
\widehat{T}=e^{-2f}(T+2if^{\overline{\gamma}}Z_{\overline{\gamma}}%
-2if^{\gamma}Z_{\gamma}),\text{ \ }\widehat{Z}_{\alpha}=e^{-f}Z_{\alpha}.
\]

Next we will derive the invariance property for the CR Paneitz operator. Let
us first recall the CR Paneitz operator as following :

\begin{definition}
Let $(M,\xi,\theta)$ be a closed pseudohermitian $(2n+1)$-manifold. Define
\[%
\begin{array}
[c]{c}%
P\varphi=\sum_{\alpha=1}^{n}(\varphi_{\overline{\alpha}}{}^{\overline{\alpha}%
}{}_{\beta}+inA_{\beta\alpha}\varphi^{\alpha})\theta^{\beta}=(P_{\beta}%
\varphi)\theta^{\beta},\text{ }\beta=1,2,\cdot\cdot\cdot,n
\end{array}
\]
which is an operator that characterizes CR-pluriharmonic functions (\cite{l2}
for $n=1$ and \cite{gl} for $n\geq2$). Here $P_{\beta}\varphi=\sum_{\alpha
=1}^{n}(\varphi_{\overline{\alpha}}{}^{\overline{\alpha}}{}_{\beta}%
+inA_{\beta\alpha}\varphi^{\alpha})$ and $\overline{P}\varphi=(\overline
{P}_{\beta}\varphi)\theta^{\overline{\beta}}$, the conjugate of $P$. Moreover,
we define
\begin{equation}
P_{0}\varphi=\delta_{b}(P\varphi)+\overline{\delta}_{b}(\overline{P}\varphi)
\label{ABC3}%
\end{equation}
which is the so-called CR Paneitz operator $P_{0}.$ Here $\delta_{b}$ is the
divergence operator that takes $(1,0)$-forms to functions by $\delta
_{b}(\sigma_{\alpha}\theta^{\alpha})=\sigma_{\alpha},^{\alpha}.$ Hence $P_{0}$
is a real and symmetric operator and
\[%
\begin{array}
[c]{c}%
\int_{M}\langle P\varphi+\overline{P}\varphi,d_{b}\varphi\rangle_{L_{\theta
}^{\ast}}d\mu=-\int_{M}\left(  P_{0}\varphi\right)  \varphi d\mu.
\end{array}
\]
We observe that (\cite{gl})
\begin{align}
P_{0}\varphi &  =\frac{1}{4}[\square_{b}\overline{\square}_{b}\varphi
-2ni(A_{\overline{\alpha}\overline{\beta}}\varphi_{\alpha})_{\beta
}]\label{2015}\\
&  =\frac{1}{4}[(\Delta_{b}^{2}+n^{2}T^{2})\varphi-2n\operatorname{Re}%
(i(A_{\overline{\alpha}\overline{\beta}}\varphi_{\alpha})_{\beta})],
\end{align}
for $\square_{b}\varphi=(\overline{\partial}_{b}\overline{\partial^{\ast}}%
_{b}+\overline{\partial^{\ast}}_{b}\overline{\partial}_{b})\varphi
=(-\Delta_{b}+inT)\varphi=-2\varphi_{\overline{\alpha}\alpha}$.
\end{definition}

\begin{remark}
\label{r1} 1. The space of kernel of the CR Paneitz operator $P_{0}$ is
infinite dimensional, containing all $CR$ -pluriharmonic functions. However,
for a closed pseudohermitian $(2n+1)$-manifold $(M,\xi,\theta)$ with $n\geq2$,
it was shown (\cite{gl}) that
\[
\ker P_{\beta}=\ker P_{0}.
\]
For a closed pseudohermitian $3$-manifold of transverse symmetry ( i.e.
vanishing torsion), we have (\cite{hi})
\[
\ker P_{1}=\ker P_{0}.
\]

2. (\cite{hi}) Let $(M^{3},\xi,\theta)$ be a pseudohermitian $3$-manifold.
Then, for rescaled contact form $\widetilde{\theta}=e^{2g}\theta,$ we have
\begin{equation}
\widetilde{P}_{1}=e^{-3g}P_{1}\ \ \text{\ \textrm{and}\ \ }\widetilde{P}%
_{0}=e^{-4g}P_{0}. \label{A}%
\end{equation}
However, the above invariance property for the CR Paneitz holds up to
lower-order in dimensions five and higher. See Corollary \ref{cTA} for details.
\end{remark}

Now we derive the following invariance property for the CR Paneitz operator.

\begin{lemma}
\label{lTA} Let $\theta$ and $\widehat{\theta}$ be contact forms on a
$(2n+1)$-dimensional pseudohermitian manifold $(M,\xi)$. If $\widehat{\theta
}=e^{2f}\theta,$ then we have%
\begin{equation}%
\begin{array}
[c]{l}%
\widehat{W}_{\alpha}-in\widehat{A}_{\alpha\beta},^{\beta}=e^{-3f}[W_{\alpha
}-inA_{\alpha\beta},^{\beta}-2(n+2)P_{\alpha}f]\\
\ \ \ \ \ \ \ \ \ \ \ \ \ \ \ \ \ \ \ \ \ \ \ \ \ \ \ +2ne^{-2f}(\widehat
{R}_{\alpha\overline{\beta}}-\frac{\widehat{W}}{n}\widehat{h}_{\alpha
\overline{\beta}})f^{\overline{\beta}}%
\end{array}
\label{4}%
\end{equation}
and
\begin{equation}%
\begin{array}
[c]{l}%
(\widehat{W}_{\alpha}-in\widehat{A}_{\alpha\beta},^{\beta}),^{\alpha}\\
=e^{-4f}[(W_{\alpha}-inA_{\alpha\beta},^{\beta}),^{\alpha}-2(n+2)P_{0}f]\\
\ \ \ \ +2n[e^{-2f}(\widehat{R}_{\alpha\overline{\beta}}-\frac{\widehat{W}}%
{n}\widehat{h}_{\alpha\overline{\beta}})f^{\overline{\beta}}],^{\alpha}.
\end{array}
\label{4a}%
\end{equation}
In particular for $n=1$, we have $\widehat{R}_{1\overline{1}}-\widehat
{W}\widehat{h}_{1\overline{1}}=0$. Then%
\begin{equation}
\widehat{W}_{1}-in\widehat{A}_{11},^{1}=e^{-3f}[W_{1}-iA_{11},^{1}-6P_{1}f]
\label{33}%
\end{equation}
and
\begin{equation}
\widehat{W}_{1}^{\ \ \ 1}-i\widehat{A}_{11},^{11}=e^{-4f}[W_{1}{}^{1}%
-iA_{11},^{11}-6P_{0}f]. \label{331}%
\end{equation}

\end{lemma}

\begin{proof}
By the contracted Bianchi identity, we have
\[%
\begin{array}
[c]{c}%
\frac{n-1}{n}(W_{\alpha}-inA_{\alpha\beta},^{\beta})=(R_{\alpha\overline
{\beta}}-\frac{W}{n}h_{\alpha\overline{\beta}}),^{\overline{\beta}}.
\end{array}
\]
Also by \cite[P 172]{l2}
\begin{equation}%
\begin{array}
[c]{c}%
(R_{\alpha\overline{\beta}}-\frac{W}{n}h_{\alpha\overline{\beta}%
})-2(n+2)(f_{\alpha\overline{\beta}}-\frac{1}{n}f_{\gamma}{}^{\gamma}%
h_{\alpha\overline{\beta}})=\widehat{R}_{\alpha\overline{\beta}}%
-\frac{\widehat{W}}{n}\widehat{h}_{\alpha\overline{\beta}}.
\end{array}
\label{5}%
\end{equation}
Following the same computation as the proof of Lemma 5.4 in \cite{hi}, \ by
using (\ref{30}), (\ref{31}) and (\ref{32}), we compute%
\begin{align*}
\widehat{W}_{\alpha} &  =\widehat{Z}_{\alpha}\widehat{W}=e^{-f}Z_{\alpha
}e^{-2f}(W-2(n+1)\Delta_{b}f-2n(n+1)|\nabla_{b}f|^{2})\\
&  =e^{-3f}[W_{\alpha}-2Wf_{\alpha}+4(n+1)(\Delta_{b}f+n|\nabla_{b}%
f|^{2})f_{\alpha}\\
&  \text{ \ \ \ \ }-2(n+1)(f_{\gamma}{}^{\gamma}{}_{\alpha}+f_{\overline
{\gamma}}{}^{\overline{\gamma}}{}_{\alpha})-4n(n+1)(f_{\gamma\alpha}f^{\gamma
}+f_{\gamma}f^{\gamma}{}_{\alpha})],\\
i\widehat{A}_{\alpha\beta},_{\overline{\gamma}} &  =i(\widehat{Z}%
_{\overline{\gamma}}\widehat{A}_{\alpha\beta}-\widehat{{\omega}}{_{\alpha}%
}^{l}(\widehat{Z}_{\overline{\gamma}})\widehat{A}_{\beta l}-\widehat{{\omega}%
}{_{\beta}}^{l}(\widehat{Z}_{\overline{\gamma}})\widehat{A}_{\alpha l})\\
&  =ie^{-f}[(Z_{\overline{\gamma}}+2f_{\overline{\gamma}})\widehat{A}%
_{\alpha\beta}+2(\delta_{\alpha\gamma}\widehat{A}_{\beta l}+\delta
_{\beta\gamma}\widehat{A}_{\alpha l})f^{l}]\\
&  =ie^{-f}(Z_{\overline{\gamma}}+2f_{\overline{\gamma}})e^{-2f}%
({A_{\alpha\beta}+2i}f_{\alpha\beta}-4if_{\alpha}f_{\beta})\\
&  +2e^{-3f}[\delta_{\beta\gamma}(i{A_{\alpha l}-2}f_{\alpha l}+4f_{\alpha
}f_{l})+\delta_{\alpha\gamma}(i{A_{\beta l}-2}f_{\beta l}+4f_{\beta}%
f_{l})]f^{l}\\
&  =e^{-3f}[i{A_{\alpha\beta,\overline{\gamma}}-2f_{\alpha\beta\overline
{\gamma}}+4(f_{\alpha\overline{\gamma}}f_{\beta}+f_{\alpha}f_{\beta
\overline{\gamma}})}]\\
&  +2e^{-3f}[\delta_{\beta\gamma}(i{A_{\alpha l}-2}f_{\alpha l}+4f_{\alpha
}f_{l})+\delta_{\alpha\gamma}(i{A_{\beta l}-2}f_{\beta l}+4f_{\beta}%
f_{l})]f^{l}.
\end{align*}
Contracting the second equation with respect to the Levi metric $\widehat
{h}_{\gamma\overline{\beta}}=h_{\gamma\overline{\beta}}$ yields%
\[%
\begin{array}
[c]{ccc}%
i\widehat{A}_{\alpha\beta},^{\beta} & = & e^{-3f}[i{A_{\alpha\beta},}^{\beta
}-2f_{\alpha\beta}{}^{\beta}+4(f_{\alpha}{}^{\beta}f_{\beta}+f_{\alpha
}f_{\beta}{}^{\beta})\\
&  & \text{ \ \ \ \ \ }{+2(n+1)(iA_{\alpha\beta}-}2{f_{\alpha\beta}%
+4}f_{\alpha}f_{\beta})f^{\beta}].
\end{array}
\]
Thus%
\[%
\begin{array}
[c]{lll}%
\widehat{W}_{\alpha}-in\widehat{A}_{\alpha\beta},^{\beta} & = & e^{-3f}%
[W_{\alpha}-in{A_{\alpha\beta},}^{\beta}-2(n+1)(f_{\beta}{}^{\beta}{}_{\alpha
}+f_{\overline{\beta}}{}^{\overline{\beta}}{}_{\alpha})+2nf_{\alpha\beta}%
{}^{\beta}\\
&  & -2Wf_{\alpha}-2n(n+1)i{A_{\alpha\beta}}f^{\beta}+4(n+1)(f_{\beta}%
{}^{\beta}+f_{\overline{\beta}}{}^{\overline{\beta}})f_{\alpha}\\
&  & -4n(n+1)f^{\beta}{}_{\alpha}f_{\beta}-4n(f_{\alpha}{}^{\beta}f_{\beta
}+f_{\beta}{}^{\beta}f_{\alpha})].
\end{array}
\]
By using the commutation relations (\cite[Lemma 2.3]{l2})
\[
-2(n+1)f_{\beta}{}^{\beta}{}_{\alpha}+2nf_{\alpha\beta}{}^{\beta
}=-2f_{\overline{\beta}}{}^{\overline{\beta}}{}_{\alpha}+2nR_{\alpha
\overline{\beta}}f^{\overline{\beta}}-2in{A_{\alpha\beta}}f^{\beta},
\]
and%
\[%
\begin{array}
[c]{c}%
{f_{\alpha\overline{\beta}}-f_{\overline{\beta}\alpha}=ih_{\alpha
\overline{\beta}}f}_{0},
\end{array}
\]
and by (\ref{5})
\[%
\begin{array}
[c]{c}%
\lbrack(R_{\alpha\overline{\beta}}-\frac{W}{n}h_{\alpha\overline{\beta}%
})-2(n+2)(f_{\alpha\overline{\beta}}-\frac{1}{n}f_{\gamma}{}^{\gamma
}{h_{\alpha\overline{\beta}}})]f^{\overline{\beta}}=e^{f}(\widehat{R}%
_{\alpha\overline{\beta}}-\frac{\widehat{W}}{n}\widehat{h}_{\alpha
\overline{\beta}})f^{\overline{\beta}},
\end{array}
\]
we obtain the following transformation law%
\[%
\begin{array}
[c]{l}%
\widehat{W}_{\alpha}-in\widehat{A}_{\alpha\beta},^{\beta}-2ne^{-2f}%
(\widehat{R}_{\alpha\overline{\beta}}-\frac{\widehat{W}}{n}\widehat{h}%
_{\alpha\overline{\beta}})f^{\overline{\beta}}\\
=e^{-3f}[W_{\alpha}-in{A_{\alpha\beta},}^{\beta}-2(n+2)(f_{\overline{\beta}}%
{}^{\overline{\beta}}{}_{\alpha}+in{A_{\alpha\beta}}f^{\beta})]\\
=e^{-3f}[W_{\alpha}-inA_{\alpha\beta},^{\beta}-2(n+2)P_{\alpha}f].
\end{array}
\]

Then (\ref{4}) and (\ref{4a}) follow easily.
\end{proof}

By applying (\ref{4}) and the method of \cite[Lemma 4.7]{hi}, we have

\begin{corollary}
\label{cTA} Let $\theta$ and $\widehat{\theta}$ be contact forms on a
$(2n+1)$-dimensional pseudohermitian manifold $(M,\xi)$. If $\widehat{\theta
}=e^{2f}\theta,$ then
\begin{equation}%
\begin{array}
[c]{ccl}%
\widehat{P}_{\alpha}f & = & e^{-3f}P_{\alpha}f\\
&  & +\frac{n}{2(n+1)}(1-e^{-2f})(\widehat{R}_{\alpha\overline{\beta}}%
-\frac{\widehat{W}}{n}\widehat{h}_{\alpha\overline{\beta}})f^{\overline{\beta
}}\\
&  & -\frac{n(n+2)}{(n+1)}(1-e^{-2f})\widehat{{B}}(f)_{\alpha\overline{\beta}%
}f^{\overline{\beta}}.
\end{array}
\label{4b}%
\end{equation}
Here $\widehat{{B}}(f)_{\alpha\overline{\beta}}=f_{\alpha\overline{\beta}%
}-\frac{1}{n}f_{\gamma}{}^{\gamma}\widehat{{h}}{_{\alpha\overline{\beta}}.}$
In particular for $n=1,$ we have $\widehat{{B}}(f)_{\alpha\overline{\beta}}=0$
and
\begin{equation}
\widehat{P}_{1}f=e^{-3f}P_{1}f\ \ \ \mathrm{and}\ \ \ \widehat{P}_{0}%
f=e^{-4f}P_{0}f\ . \label{33b}%
\end{equation}

\end{corollary}

\section{The Proof of Main Theorem}

Let $(M,\xi,J_{0},\theta_{0})$ be a closed pseudohermitian $(2n+1)$-manifold.
As in (\ref{2.5}) below, we do not know the subellipticity for the highest
weight operator $L_{J,\theta}(G)$ of the linearization $\delta(2Q(J,\theta))$
for the torsion \ flow (\ref{1}). \ Instead we consider the CR torsion flow
(\ref{2015BBB}), with the particular choice of the contact form $\theta_{0}$
in $\xi$ and$\ \gamma_{0}$ as in (\ref{2018AAA1}), in which the highest weight
operator $L_{J,\theta}(G)$ of the linearization $\delta(2Q(J,\theta))$ as in
(\ref{2015A}) is subelliptic. We refer to the \textbf{step I} of the proof of
Theorem \ref{TA} below for details.

In the following, we first set up the CR torsion flow (\ref{2015BBB}) from the
point view of finding its  gauge-fixing condition or integrability condition
which is a crucial step for the subellipticity of the linearization
(\ref{2015A}). \ More precisely, we consider the following flow
\begin{subequations}
\begin{equation}
\left\{
\begin{array}
[c]{l}%
\frac{\partial J}{\partial t}=2F_{J,\theta},\\
\frac{\partial\theta}{\partial t}=2\eta\theta,
\end{array}
\right.  \label{2015c}%
\end{equation}
on $(M,\xi,J,\theta,\varphi)\times\lbrack0,T)$ with
\end{subequations}
\[
F_{J,\theta}:=F_{\alpha}{}^{\overline{\beta}}\theta^{\alpha}\otimes
Z_{\overline{\beta}}+F_{\overline{\alpha}}{}^{\beta}\theta^{\overline{\alpha}%
}\otimes Z_{\beta}%
\]
and
\[
\eta:=-e^{\varphi}[W+(n+1){\Delta_{b}}\varphi-n(n+1)\varphi_{\alpha}{}%
\varphi{}^{\overline{\alpha}}].
\]
Here $J=i\theta^{\alpha}\otimes Z_{\alpha}-i\theta^{\overline{\alpha}}\otimes
Z_{\overline{\alpha}}$ and $F_{\alpha}{}^{\overline{\beta}}:=e^{\varphi
}[A_{\alpha}{}^{\overline{\beta}}-i\varphi_{\alpha}{}^{\overline{\beta}%
}-i\varphi_{\alpha}{}\varphi{}^{\overline{\beta}}].$ We come out with the
following flow
\begin{equation}
\left\{
\begin{array}
[c]{l}%
\frac{\partial J}{\partial t}=2e^{\varphi}(A_{\alpha}{}^{\overline{\beta}%
}-i\varphi_{\alpha}{}^{\overline{\beta}}-i\varphi_{\alpha}{}\varphi
{}^{\overline{\beta}})\theta^{\alpha}\otimes Z_{\overline{\beta}}+Conj,\\
\frac{\partial\theta}{\partial t}=-2e^{\varphi}[W+(n+1){\Delta_{b}}%
\varphi-n(n+1)\varphi_{\overline{\alpha}}{}\varphi{}^{\overline{\alpha}%
}]\theta,
\end{array}
\right.  \label{2015bb}%
\end{equation}
on $(M,\xi,J,\theta)\times\lbrack0,T).$ \ Thus, for \ $\theta=e^{2\lambda
}\theta_{0},$ we denote
\[
\lambda^{\perp}:=\lambda-\lambda^{\ker}%
\]
with
\[
\lambda^{\ker}\in\ker\overset{0}{P_{\beta}}\text{ \ \textrm{i.e.\ }\ }%
\overset{0}{P_{\beta}}(\lambda^{\ker})=0.
\]
Then, by choosing $\varphi=2\lambda^{\perp}$, the torsion $\widetilde
{A_{\alpha}{}^{\overline{\beta}}}$ and scalar curvature $\widetilde{W}$ with
respect to
\[
\widetilde{\theta}=e^{-\varphi}\theta=e^{-\varphi}e^{2\lambda}\theta
_{0}=e^{2\lambda^{\ker}}\theta_{0}%
\]
will be
\[
\left\{
\begin{array}
[c]{ccl}%
\widetilde{A_{\alpha}{}^{\overline{\beta}}} & = & e^{\varphi}(A_{\alpha}%
{}^{\overline{\beta}}-i\varphi_{\alpha}{}^{\overline{\beta}}-i\varphi_{\alpha
}{}\varphi{}^{\overline{\beta}}),\\
\widetilde{W} & = & e^{\varphi}(W+(n+1){\Delta_{b}}\varphi-n(n+1)\varphi
_{\overline{\alpha}}{}\varphi{}^{\overline{\alpha}}).
\end{array}
\right.
\]
It follows that (\ref{2015bb}) implies
\begin{equation}
\left\{
\begin{array}
[c]{l}%
\frac{\partial J}{\partial t}=2\widetilde{A}_{J,\widetilde{\theta}},\\
\frac{\partial\theta}{\partial t}=-2\widetilde{W}\theta,\ \ \\
\widetilde{\theta}=e^{2\lambda^{\ker}}\theta_{0},\ \ \theta=e^{2\lambda}%
\theta_{0}.
\end{array}
\right.  \label{3b}%
\end{equation}
In particular%
\[
\frac{\partial\lambda^{\ker}}{\partial t}=-\widetilde{W}^{\ker}%
\text{\ \ \ \textrm{and}\ \ }\frac{\partial\lambda^{\perp}}{\partial
t}=-\widetilde{W}^{\perp}.
\]
Furthermore, we compute for $\widetilde{\theta}=e^{-\varphi}\theta$%
\[%
\begin{array}
[c]{ccl}%
\frac{\partial\widetilde{\theta}}{\partial t} & = & -2(\widetilde{W}+\frac
{1}{2}\varphi_{t})\widetilde{\theta}\\
& = & -2(\widetilde{W}+\frac{\partial\lambda^{\perp}}{\partial t}%
)\widetilde{\theta}\\
& = & -2(\widetilde{W}-\widetilde{W}^{\perp})\widetilde{\theta}.
\end{array}
\]
Hence, with the choice of $\varphi(x,t)=2\lambda^{\perp}(x,t),$ (\ref{2015bb})
is equivalent to the following modified torsion flow \ on $M\times\lbrack0,T)$
with the maximal time $T:$
\begin{equation}
\left\{
\begin{array}
[c]{l}%
\frac{\partial J}{\partial t}=2\widetilde{A}_{J,\widetilde{\theta}},\\
\frac{\partial\widetilde{\theta}}{\partial t}=-2(\widetilde{W}-\widetilde
{W}^{\perp})\widetilde{\theta},\ \text{\textrm{\ }}\\
\text{\textrm{\ }}\widetilde{\theta}=e^{2\lambda^{\ker}}\theta_{0}%
\end{array}
\right.  \label{3}%
\end{equation}
coupled with
\begin{equation}
\frac{\partial\lambda^{\perp}}{\partial t}=-\widetilde{W}^{\perp}.\label{3a}%
\end{equation}

According to (\ref{3}) which is the deformation with respect to $e^{2\lambda
^{\ker}}\theta_{0}$ only, then we in particular consider the original modified
torsion flow (\ref{2015bb}) with the special choice
\begin{equation}
\frac{1}{2}\varphi(x,t)=\lambda^{\perp}(x,t)\text{ \ \ \ and \ \ }%
\ \ \lambda^{\perp}(x,0)=\widetilde{W}^{\perp}(x,0). \label{2018A}%
\end{equation}
Then if \textbf{the initial Tanaka-Webster scalar curvature }$\widetilde
{W}(x,0)$ w.r.t. $e^{\gamma_{0}}\theta_{0}$ is \textbf{pluriharmonic, i.e. }
\[
\widetilde{W}^{\perp}(x,0)=0,
\]
we have
\[
\text{\ }\widetilde{W}^{\perp}(x,t)=0=\lambda^{\perp}(x,t)
\]
as long as the solution exists for all $t\in\lbrack0,T)$. \ In this case, it
follows that (\ref{2015bb}) is equivalent to the following torsion flow \ on
$M\times\lbrack0,T)$ with the maximal time $T:$ \textrm{\ }%
\begin{equation}
\left\{
\begin{array}
[c]{l}%
\frac{\partial J}{\partial t}=2\widetilde{A}_{J,\widetilde{\theta}},\\
\frac{\partial\widetilde{\theta}}{\partial t}=-2\widetilde{W}\widetilde
{\theta},\\
\widetilde{\theta}=e^{2\gamma}\theta_{0};\ \overset{0}{P_{\beta}}(\gamma
_{0})=0=\widetilde{W}^{\perp}(0).
\end{array}
\right.  \label{2018}%
\end{equation}

Finally by replacing $\widetilde{\theta}$ by $\theta$ without ambiguity, we
rewrite (\ref{2018}) as the following tosion flow on $(M,\xi,J,\theta
)\times\lbrack0,T):$
\begin{equation}
\left\{
\begin{array}
[c]{l}%
\frac{\partial J}{\partial t}=2A_{J,\theta},\\
\frac{\partial\theta}{\partial t}=-2W\theta,\ \\
\theta=e^{2\gamma}\theta_{0};\ \overset{0}{P_{\beta}}(\gamma_{0})=0=W^{\perp
}(0).
\end{array}
\right.  \label{2015b}%
\end{equation}
Note that
\[
W^{\perp}(x,t)=0=\gamma^{\perp}(x,t)
\]
as long as the solution exists for all $t\in\lbrack0,T)$.

Next for $\theta=e^{2\lambda}\theta_{0},$ it follows from Lemma \ref{lTA},
Corollary \ref{cTA} and commutation relations, we compute%
\[%
\begin{array}
[c]{ccl}%
\eta_{\beta}+inF_{\beta}{}^{\overline{\alpha}},_{\overline{\alpha}} & = &
-W_{\beta}+inA_{\beta}{}^{\overline{\alpha}},_{\overline{\alpha}}\\
& = & -e^{-3\gamma}[\overset{0}{W}_{\beta}-in\overset{0}{A}_{\beta}%
{}^{\overline{\alpha}},_{\overline{\alpha}}-2(n+2)\overset{0}{P}_{\beta}%
\gamma]\\
&  & -2ne^{-2\gamma}(R_{\alpha\overline{\beta}}-\frac{W}{n}h_{\alpha
\overline{\beta}})\gamma^{\overline{\beta}}\\
& = & -e^{-3\gamma}[\overset{0}{W}_{\beta}-in\overset{0}{A}_{\beta}%
{}^{\overline{\alpha}},_{\overline{\alpha}}]\\
&  & -2ne^{-2\gamma}(R_{\alpha\overline{\beta}}-\frac{W}{n}h_{\alpha
\overline{\beta}})\gamma^{\overline{\beta}}.
\end{array}
\]
Hence
\begin{equation}
\eta_{\beta}+inF_{\beta}{}^{\overline{\alpha}},_{\overline{\alpha}%
}=\mathrm{l.o.t.in}\ \gamma.\label{Bianchi}%
\end{equation}
We observe that (\ref{Bianchi}) will be the gauge-fixing condition or
integrability condition for the flow (\ref{2015b}) below.

We are ready to prove the existence of short-time solution for the torsion
flow (\ref{2015BBB}) on $(M,\xi,J,\theta)\times\lbrack0,T):$

\begin{proof}
[Proof of Theorem \ref{TA}]{\noindent\textbf{Step I.} \textbf{We first}
\textbf{find the linearization of the flow}} (\ref{2015BBB}) :

For\ $\ \theta=e^{2\gamma}\theta_{0}$ with $\overset{0}{P_{\beta}}(\gamma)=0$,
we can rewrite the flow (\ref{2015BBB}) as
\begin{equation}%
\begin{array}
[c]{c}%
\frac{\partial}{\partial t}(J\oplus\theta)=2Q(J,\theta).
\end{array}
\label{2.0}%
\end{equation}
Here \ $Q(J,\theta):=F_{J,\theta}\oplus\eta\theta$ \ with
\[%
\begin{array}
[c]{ccl}%
F_{J,\theta} & := & A_{J,\theta},\\
\eta & := & -W.
\end{array}
\]
with $\overset{0}{P_{\beta}}(W)=0$. We use $\delta J,$ $\delta\theta$ to
denote the variations of $J$ and $\theta$, respectively. Set
\[
\delta J=2E\ \ \mathrm{and}\ \ \delta\theta=2h\theta
\]
with $\overset{0}{P}_{\beta}(h)=0.$ Here $E$ is an endomorphism:
$\xi\rightarrow\xi$ satisfying $J\circ E+E\circ J=0$ and $h$ is a smooth
function. Let $\delta_{J}$ and $\delta_{\theta}$ denote the linearization
operators with respect to $J$ and $\theta$. From section $3$, we see that%
\begin{equation}%
\begin{array}
[c]{ccl}%
\delta_{J}A_{\alpha\beta} & = & i(E_{\alpha\beta},_{0})\\
\delta_{\theta}A_{\alpha\beta} & = & 2(ih_{\alpha\beta})+l.o.t.
\end{array}
\label{2.1}%
\end{equation}
and%
\begin{equation}%
\begin{array}
[c]{ccl}%
\delta_{J}W & = & i(E{_{\gamma\alpha}},^{\gamma\alpha}-{E{_{\overline{\gamma
}\overline{\alpha}}},^{\overline{\gamma}\overline{\alpha}}})+l.o.t.\\
\delta_{\theta}W & = & -2(n+1)({\Delta_{b}h)}+l.o.t.
\end{array}
\label{2.2}%
\end{equation}
with $\overset{0}{P_{\beta}}(\gamma)=0$ and $\overset{0}{P}_{\beta}(h)=0.$ We
first compute the linearization of $2A_{J,\theta}$. Let $\mathcal{O}_{m}$
denote an operator of weight $\leq m$ (see \cite[page 234]{cl1}), and define
the total variation as $\delta=\delta_{J}+\delta_{\theta}$. Put
\[
G:=E\oplus h\theta.
\]
With (\ref{2.1}) we compute the variation of the torsion as
\begin{equation}
\delta(2A_{J,\theta})=2\operatorname{Re}\left(  [2iE{_{\alpha}}^{\overline
{\beta}},_{0}+4ih_{\alpha}{}^{\overline{\beta}}]\theta^{\alpha}\otimes
Z_{\overline{\beta}}\right)  +\mathcal{O}_{1}(G).\label{2.4a}%
\end{equation}
The linearization of $-2W\theta$ can be computed with (\ref{2.2}),%
\begin{equation}%
\begin{array}
[c]{ccl}%
\delta(-2W\theta) & = & \{2i({E{_{\overline{\gamma}\overline{\alpha}}%
},^{\overline{\gamma}\overline{\alpha}}-}E{_{\gamma\alpha}},^{\gamma\alpha})\\
&  & +4(n+1)(\Delta_{b}h)+\mathcal{O}_{1}(G)\}\theta
\end{array}
\label{2.4}%
\end{equation}
With these linearizations, we write the highest weight operator $L_{J,\theta
}(G)$ of the linearization $\delta(2Q(J,\theta))$ as%
\begin{equation}%
\begin{array}
[c]{l}%
\ L_{J,\theta}(G)\\
=2\operatorname{Re}\left(  \{[2iE{_{\alpha}}^{\overline{\beta}},_{0}%
+4ih_{\alpha}{}^{\overline{\beta}}]\}\theta^{\alpha}\otimes Z_{\overline
{\beta}}\right)  \\
\ \ \oplus\lbrack2i({E{_{\overline{\beta}\overline{\alpha}}},^{\overline
{\beta}\overline{\alpha}}-}E{_{\beta\alpha}},^{\beta\alpha})+4(n+1)\Delta
_{b}h]\theta.
\end{array}
\label{2.5}%
\end{equation}

{\noindent Next }we define the linear operator $H(J,\theta)$ by%
\begin{equation}
H(J,\theta)(G):=(h_{\alpha}+inE_{\alpha\beta},^{\beta})-\mathcal{O}_{0}(G)
\label{eq:def_H}%
\end{equation}
with $G=E\oplus h\theta$ and $h\in\ker(\overset{0}{P}_{\beta}).$ Here
$\mathcal{O}_{0}(G)$ is the lower-order terms as in (\ref{Bianchi}).

Thus $Q(J,\theta)$ satisfies the condition as in (\ref{Bianchi}) :%
\[
H(J,\theta)Q(J,\theta)=(\eta_{\alpha}+inF_{\alpha}{}^{\overline{\beta}%
},_{\overline{\beta}})-\mathcal{O}_{0}(G)=0.
\]

Now from the following commutation relations
\[%
\begin{array}
[c]{l}%
E_{\alpha\beta},^{\beta}{}_{\gamma}-E_{\alpha\gamma},^{\beta}{}_{\beta
}=i(n-1)E_{\alpha\gamma,0}+R_{\alpha\overline{\sigma}}E^{\overline{\sigma}}%
{}_{\gamma}-R_{\alpha}{}^{\beta}{}_{\gamma\overline{\sigma}}E^{\overline
{\sigma}}{}_{\beta},\\
E_{\alpha\gamma},_{\beta}{}^{\beta}-E_{\alpha\gamma},^{\beta}{}_{\beta
}=inE_{\alpha\gamma,0}+R_{\gamma}{}^{\sigma}E_{\alpha\sigma}+R_{\alpha}%
{}^{\sigma}E_{\sigma\gamma}.
\end{array}
\]
We derive that on the subspace
\begin{equation}
\ker H(J,\theta)=\{G=E\oplus h\theta|\ (h_{\alpha}+inE_{\alpha\beta},^{\beta
})-\mathcal{O}_{0}(G)=0\}\label{2.9}%
\end{equation}
the following identities hold
\begin{equation}%
\begin{array}
[c]{rcl}%
(\Delta_{b}h) & = & in(E_{\overline{\alpha}\overline{\beta}},^{\overline
{\beta}\overline{\alpha}}-E_{\alpha\beta},^{\beta\alpha})+\mathcal{O}%
_{1}(G),\\
(\Delta_{b}E_{\alpha}{}^{\overline{\beta}}) & = & [i(2-n)E_{\alpha}%
{}^{\overline{\beta}},_{0}+\frac{2}{n}ih_{\alpha}{}^{\overline{\beta}%
}]+\mathcal{O}_{1}(G).
\end{array}
\label{2.8a}%
\end{equation}
Here we have used the gauge-fixing condition
\[
(h_{\alpha}+inE_{\alpha\beta},^{\beta})=\mathcal{O}_{0}(G).
\]

Therefore it follows from (\ref{2.5}) that the highest weight operator
$L_{J,\theta}(G)$ of the linearization $2\delta(Q(J,\theta))$ is%
\begin{equation}%
\begin{array}
[c]{l}%
L_{J,\theta}(E\oplus h\theta)\\
=\{4n\operatorname{Re}[(\mathcal{L}_{\alpha}E_{\alpha}{}^{\overline{\beta}%
}))\theta^{\alpha}\otimes Z_{\overline{\beta}}]\oplus2(2n+2+\frac{1}%
{n})({\Delta}_{b}h)\theta\}\\
=\{4n\operatorname{Re}[\mathcal{L}_{\alpha}(E_{\alpha}{}^{\overline{\beta}%
})\theta^{\alpha}\otimes Z_{\overline{\beta}}]\oplus2(2n+2+\frac{1}%
{n})({\Delta}_{b}h)\theta\}\\
=\{4n\operatorname{Re}[\mathcal{L}_{\alpha}(E_{\alpha}{}^{\overline{\beta}%
})\theta^{\alpha}\otimes Z_{\overline{\beta}}]\oplus2(2n+2+\frac{1}%
{n})({\Delta}_{b}h)\theta\}
\end{array}
\label{2015A}%
\end{equation}
on $\ker H(J,\theta)(E\oplus h\theta)$ with
\begin{equation}
\mathcal{L}_{\alpha}={\Delta}_{b}-i\alpha T,\ \ \ \alpha=-\frac{(n-1)^{2}}%
{n}=-(n-2+\frac{1}{n}). \label{2018a}%
\end{equation}

\textbf{Step 2. The sub-ellipticity and short-time existence :}

It follows from (\ref{2018a}) that the Folland-Stein operator $\mathcal{L}%
_{\alpha}$ with this particular value
\[
\alpha=-(n-2+\frac{1}{n})
\]
is subelliptic (see \cite{fs}). \ In particular $\alpha=0$ if $n=1$ and then
$\mathcal{L}_{0}=\Delta_{b}$ which is subelliptic. Moreover, it follows from
(\ref{Bianchi}) and (\ref{2.6}), (\ref{2.6a}) below that we may then use (see
remark \ref{r2})
\begin{equation}
H(J,\theta)(G)=((h_{\alpha}+inE_{\alpha\beta},^{\beta})-\mathcal{O}%
_{0}(G)=0.\label{2.19}%
\end{equation}
as the gauge-fixing condition (see appendix ) for the torsion flow
(\ref{2015BBB}). It follows that $L$ as shown in (\ref{2015A}) maps $\ker
H(J,\theta)$ into itself modulo terms of lower weight. Since $L$ is
subelliptic on $\ker H(J,\theta)$, by comparing \cite[Lemma 3.3]{cl1}, once
the highest weight operator $L_{J,\theta}(G)$ of the linearization
$2\delta(Q(J,\theta))$ on $\ker H(J,\theta)$ is subelliptic, then the
short-time existence of the torsion flow (\ref{2015BBB}) follows easily from
the linear theory of the Folland-Stein space and section $5$ as in \cite{cl1}
under the gauge-fixing condition (\ref{2.19}). For completeness, we will
discuss the gauge-fixing condition more details as in Appnedix for $n=1$.
\end{proof}

{\noindent}

To make the proof more clear, we make a remark as following :

\begin{remark}
\label{r2} As shown in the appendix (\ref{2.6}), $\ker\tilde{B}_{J}$ is the
natural gauge-fixing condition for
\[
\left\{
\begin{array}
[c]{l}%
\frac{\partial J}{\partial t}=2A_{J,\theta},\\
\frac{\partial\theta}{\partial t}=-2W\theta,\\
\text{\textrm{\ }}\theta=e^{2\lambda}\theta_{0}.
\end{array}
\right.
\]
However, we do not know the subellipticity for the highest weight operator
$L_{J,\theta}(G)$ of the linearization $\delta(2Q(J,\theta))$ as in
(\ref{2.5}). Instead, it follows from (\ref{Bianchi}) that (\ref{2.19}) is the
right gauge-fixing condition for the current type CR torsion flow
\[
\left\{
\begin{array}
[c]{l}%
\frac{\partial J}{\partial t}=2A_{J,\theta},\\
\frac{\partial\theta}{\partial t}=-2W\theta,\ \\
\theta=e^{2\gamma}\theta_{0};\ \overset{0}{P_{\beta}}(\gamma_{0})=0=W^{\perp
}(0).
\end{array}
\right.
\]
As we have proved that the highest weight operator $L_{J,\theta}(G)$ of the
linearization $\delta(2Q(J,\theta))$ as in (\ref{2015A}) is subelliptic.
\end{remark}

\appendix

\section{ A\textbf{lternative Proof via the Lions-Lax-Milgram Theorem }
\textbf{ \ }}

A. \textbf{F}{\textbf{inding the gauge-fixing condition }}\textbf{for }$n=1$
{\textbf{:}}

\ For simplicity, let $(M,\xi,J_{0},\theta_{0})$ be a closed pseudohermitian
$3$-manifold. We first observe that the highest weight operator $L_{J,\theta
}(G)$ is self-adjoint, i.e., $L^{\ast}$ $=$ $L$ with respect to the inner
product:%
\begin{align}
&  <E+h\theta,F+k\theta>\label{AAA}\\
&  =\int_{M}[E_{11}F_{\bar{1}\bar{1}}+F_{11}E_{\bar{1}\bar{1}}+2hk]\theta
\wedge d\theta.\nonumber
\end{align}
Let $X_{f}$ be the contact vector field for the real-valued function $f\in
C^{2}(M)$. Then one has
\[
\mathcal{L}_{X_{f}}\theta=-(Tf)\theta,
\]
and
\[
L_{X_{f}}J=2B_{J}^{\prime}f:=2(f^{\overline{1}}{}_{1}+iA_{1}{}^{\overline{1}%
})\theta^{1}\otimes Z_{\overline{1}}+2(f^{1}{}_{\overline{1}}-iA_{\overline
{1}}{}^{1})\theta^{\overline{1}}\otimes Z_{1}%
\]
See for instance \cite{cacc}. In particular, we have $X_{f}=T$ for $f=1$, and
the above equation reduces to
\[
\mathcal{L}_{T}J=2JA_{J,\theta}.
\]

Recall that $X_{f}$ is an infinitesimal contact diffeomorphism if and only if
$L_{X_{f}}\theta=\mu\theta$ for some function $\mu$. So an infinitesimal
contact orbit reads%
\[
L_{X_{f}}J+L_{X_{f}}\theta=2B_{J}^{\prime}f-(Tf)\theta:=2\tilde{B}_{J}%
^{\prime}f.
\]
An orthogonal infinitesimal slice is described by $\tilde{B}_{J}(G)$ where
$\tilde{B}_{J}$ is the adjoint of $\tilde{B}_{J}^{\prime}$ with respect to the
inner product (\ref{AAA}). By a direct computation (\cite[Lemma 3.4]{cl2}),
one obtains%
\begin{equation}%
\begin{array}
[c]{cll}%
\tilde{B}_{J}(G) & = & B_{J}(E)+h_{,0}\\
& = & E_{11,\bar{1}\bar{1}}+E_{\bar{1}\bar{1},11}+h_{,0}+iA_{11}E_{\bar{1}%
\bar{1}}-iA_{\bar{1}\bar{1}}E_{11}\\
& = & E_{11,\bar{1}\bar{1}}+E_{\bar{1}\bar{1},11}+h_{,0}+\mathcal{O}_{1}(E,h).
\end{array}
\label{2.6}%
\end{equation}
Observe that $F_{J,\theta}\oplus\eta\theta$ satisfies%
\[
(\eta_{1}+iF_{1}{}^{\overline{1}},_{\overline{1}})-\mathcal{O}_{0}(F,\eta)=0
\]
and then
\begin{align}
&  \tilde{B}_{J}(Q)\label{2.6a}\\
\  &  =\tilde{B}_{J}(F_{J,\theta}\oplus\eta\theta)\nonumber\\
&  =F_{11,\bar{1}\bar{1}}+F_{\bar{1}\bar{1},11}+\eta,_{0}+\mathcal{O}%
_{1}(F,\eta)=0.\nonumber
\end{align}
by the CR Bianchi identity. \ On the other hand, it follows from (\ref{2.9})
that
\[
(E_{11,\bar{1}\bar{1}}+E_{\bar{1}\bar{1},11}+h_{,0})+\mathcal{O}_{1}(E,h)=0
\]
on $\ker H(J,\theta).$ It follows from (\ref{Bianchi}) that we may then use
(see remark \ref{r2})
\[
H(J,\theta)(G)=((h_{1}+inE_{11},^{1})-\mathcal{O}_{0}(G)=0.
\]
as the gauge-fixing condition for the torsion flow (\ref{2015BBB}).

\bigskip

\textbf{B.} \textbf{Alternative proof of the short-time existence via the
Lions-Lax-Milgram theorem for }$n=1$\textbf{ }{\textbf{:}} \

\begin{proof}
We claim that $L$ is subelliptic on $\ker H(J,\theta)(G)$. Note that $L^{\ast
}L=L^{2}$. We denote that
\begin{equation}
L^{\ast}L(E+h\theta)=F+k\theta\label{2.7}%
\end{equation}
with
\begin{equation}
F=4({\Delta}_{b}^{2}E_{11}\theta^{1}\otimes Z_{\bar{1}}+{\Delta}_{b}%
^{2}E_{\bar{1}\bar{1}}\theta^{\bar{1}}\otimes Z_{1}):=4{\Delta}_{b}^{2}E
\label{2.12}%
\end{equation}
and
\begin{equation}
k=100{\Delta}_{b}^{2}h. \label{2.13}%
\end{equation}
Substituting (\ref{2.12}) and (\ref{2.13}) into (\ref{2.7}), we compute%
\begin{equation}%
\begin{array}
[c]{ccl}%
||L(E+h\theta)||_{L^{2}}^{2} & = & <L^{\ast}L(E+h\theta),E+h\theta>\\
& = & <4{\Delta}_{b}^{2}E,\ E>+200\int_{M}(\Delta_{b}^{2}h)h\theta\wedge
d\theta\\
& = & 8||{\Delta}_{b}E||_{L^{2}}^{2}+200||\Delta_{b}h||_{L^{2}}^{2}.
\end{array}
\label{2.14}%
\end{equation}
Since $\Delta_{b}$ is subelliptic, we conclude that
\begin{equation}
||E||_{S_{2}}^{2}+||h||_{S_{2}}^{2}\leq C(||L(E+h\theta)||_{L^{2}}%
^{2}+||E+h\theta||_{L^{2}}^{2}) \label{2.17}%
\end{equation}
for some constant $C$ $>$ $0,$ where $||$ $\cdot$ $||_{S_{2}}$ denotes the
norm on the Folland-Stein space $S_{2}$ $:=$ $\{(E,h)$ $\in$ $L^{2}$ $|$
$P(E,h)$ $\in$ $L^{2}$ whenever $w(P)$ $\leq$ $2\}$ as in (\ref{KKK}).
Finally, we can obtain estimates of higher order derivatives from (\ref{2.17})
and interpolation inequalities in the Folland-Stein spaces $S_{k}$ by
observing that $[L,$ $\nabla_{Z_{1}}]$ and $[L,$ $\nabla_{Z_{\bar{1}}}]$ are
operators of weight $2$ :%
\begin{equation}
||E||_{S_{k+2}}^{2}+||h||_{S_{k+2}}^{2}\leq C(||L(E+h\theta)||_{S_{k}}%
^{2}+||E+h\theta||_{L^{2}}^{2}) \label{2.18}%
\end{equation}
for $E,$ $h$ satisfying the gauge-fixing condition.

On the other hand%
\begin{equation}%
\begin{array}
[c]{ccl}%
<L(E+h\theta),\ E+h\theta> & = & <2{\Delta}_{b}E,\ E>+20\int_{M}(\Delta
_{b}h)h\theta\wedge d\theta\\
& = & -2\int_{M}|\nabla_{b}E|^{2}\theta\wedge d\theta-20\int_{M}|\nabla
_{b}h|^{2}\theta\wedge d\theta\\
& = & -2||\nabla_{b}E||_{L^{2}}^{2}-20||\nabla_{b}h||_{L^{2}}^{2}.
\end{array}
\label{2.18a}%
\end{equation}
It follows from (\ref{2.18a}) and the Poincare-type inequality (\cite{j})
that
\begin{equation}
-<L(G),\ G>\geq C(||\nabla_{b}G||_{L^{2}}^{2}+||G||_{L^{2}}^{2}) \label{2.20}%
\end{equation}
on $\ker H(J,\theta)(G)$. Hence $<L(\cdot),\ \cdot$\ $>$ is coercive
(\cite{gt}). \

Note that Theorem 21.1 from \cite{fs} gives the regularity in space and
Theorem 4.6 in \cite{cl1} gives the regularity in time.

Let $J_{0}$ be any $C^{\infty}$\ smooth oriented $CR$\ structure compatible
with $\xi$ and $\theta_{0}$\ be any $C^{\infty}$\ smooth oriented contact form
for $\xi.$\ Based on the subelliptic estimates (\ref{2.18}), the coercive
property (\ref{2.20}) plus the implicit function theorem, we have, for any
integer $m,$\textit{\ }there exists $\delta>0$ and a unique $C^{m}$-solution
$J(t)$\ and $\theta(t)$ to the torsion flow~(\ref{2015BBB}) on the interval
$[0,\delta)$ such that $(J(0),\theta(0)\,)=(J_{0},e^{2\gamma(0)}\theta_{0})$.
This is the so-called Lions-Lax-Milgram theorem as in \cite[Theorem 5.8]{gt},
\cite{tre}, \cite{y}, \cite{lm} for elliptic-type and Theorem 4.1. in chapter
IV of \cite{s} and a generalization due to J. L. Lions of the Lax-Milgram
Thoerem of \cite[Lemma 41.2]{tre} for parabolic-type.
\end{proof}

\end{document}